\documentclass[10pt,oneside,a4paper]{amsart}
\usepackage{latexsym}
\usepackage{amsfonts,amsmath,amssymb,indentfirst}
\usepackage[english]{babel}
\usepackage[all]{xy}
\usepackage[pdftex]{hyperref}
\newcommand{\bdism}{\begin{displaymath}}
\newcommand{\edism}{\end{displaymath}}

\newcommand{\rr}{\mathbb{R}}
\newcommand{\qq}{\mathbb{Q}}

\newcommand{\nn}{\mathbb{N}}
\newcommand{\pp}{\mathbb{P}}

\DeclareMathOperator{\Vol}{Vol}

\DeclareMathOperator{\NE}{\overline{NE}}
\newtheorem{theorem}{Theorem}[section]
\newtheorem{proposition}[theorem]{Proposition}

\newtheorem{lemma}[theorem]{Lemma}
\newtheorem{remark}[theorem]{Remark}

\newtheorem{definition}[theorem]{Definition}
\newtheorem{question}[theorem]{Question}
\newtheorem{conjecture}[theorem]{Conjecture}

\address{Department of Mathematics, Columbia University, New York NY 10027,
USA} \email{dicerbo@math.columbia.edu}

\address{Department of Mathematics, Notre Dame University, Notre Dame IN 46556,
USA}
\curraddr{Max Planck Institute for Mathematics, Vivatsgasse 7, 53111 Bonn, Germany}
\email{luca@mpim-bonn.mpg.de}

\author{Gabriele Di Cerbo and Luca F. Di Cerbo}

\title{On the canonical divisor of smooth toroidal compactifications}

\begin{document}

\maketitle
\pagestyle{headings}
\begin{abstract}
In this paper, we show that the canonical divisor of a smooth toroidal compactification of a complex hyperbolic manifold must be nef if the dimension is greater or equal to three. Moreover, if $n\geq 3$ we show that the numerical dimension of the canonical divisor of a smooth $n$-dimensional compactification is always bigger or equal to $n-1$. We also show that up to a finite \'etale cover all such compactifications have ample canonical class,  therefore refining a classical theorem of Mumford and Tai.  Finally, we improve in all dimensions $n\geq 3$ the cusp count for finite volume complex hyperbolic manifolds given in \cite{DiCerbo2}. 
\end{abstract}

\tableofcontents

\section{Introduction}
\markboth{Right}{\textbf{ON THE CANONICAL DIVISOR OF TOROIDAL COMPACTIFICATIONS}}
\pagenumbering{arabic}

In 1984 Hirzebruch constructed the first examples of smooth compactifications 
of complex hyperbolic manifolds with non-nef canonical divisors.
The surfaces constructed in \cite{Hirzebruch} are
blow-ups of a particular Abelian surface at certain configurations
of points. The construction given by Hirzebruch in \cite{Hirzebruch} is quite simple and elegant and it is of interest to determine whether or not this construction is generalizable
to higher dimensions. In fact, one of the important aspects of Hirzebruch's construction is that it explicitly provides a class of  
concrete examples in a field, complex hyperbolic geometry, where explicit examples are usually hard to find. Thus, any generalization of such a construction would be most welcome. One of the goals of this paper is to show that Hirzebruch's construction is peculiar to complex dimension two and it does not admit higher dimensional generalizations. Interestingly, this negative result follows from a fundamental difference between complex hyperbolic geometry in dimension two versus higher dimensions. The main result of the paper is the following theorem. 

\begin{theorem}\label{nef}
Let $(X,D)$ be a smooth toroidal compactification of a complex hyperbolic manifold of dimension $n\geq 3$. Then 
$K_{X}+\alpha D$ is ample for any $\alpha\in(0, 1)$. In particular, $K_{X}$ is always nef.
\end{theorem}

Theorem \ref{nef} is unexpected since it implies that the theory
of smooth toroidal compactifications of ball quotients is somewhat easier in dimensions $\geq 3$.  
In particular, these varieties appear to be very simple from the
minimal model point of view, which is quite unusual in higher dimensions.\\

The construction of Hirzebruch is important for a second reason, namely it provides an infinite sequence of distinct smooth compactifications of Kodaira dimension zero. Recall that the problem of determining the Kodaira dimension of such smooth compactifications is a central problem in the theory of compactifications of locally symmetric varieties.  On this problem a fundamental result of Mumford and Tai ensures that a ``generic'' smooth compactification is of general type.

\begin{theorem}[Mumford-Tai \cite{Mumford}]
Let $\Gamma$ be a neat arithmetic group acting on a bounded symmetric domain $\mathcal{B}^{n}$ with $n\geq 2$. There exists a finite index subgroup $\Gamma_{0}\leq \Gamma$ such that for all $\Gamma_{1}\leq \Gamma_{0}$, the smooth compactification of $\mathcal{B}/\Gamma_{1}$, say $(X, D)$, is of general type. In other words, $K_{X}$ is big.
\end{theorem}

As a first application of Theorem \ref{nef}, we can improve Mumford's theorem when the bounded locally symmetric domain is the complex hyperbolic space $\mathcal{H}^{n}$.

\begin{theorem}\label{Mum}
Let $\Gamma$ be a torsion free lattice acting on $\mathcal{H}^{n}$ with $n\geq 2$. There exists a finite index subgroup $\Gamma_{0}\leq \Gamma$ such that for all $\Gamma_{1}\leq \Gamma_{0}$, the smooth compactification of $\mathcal{H}/\Gamma_{1}$, say $(X, D)$, has ample canonical class. In other words, $K_{X}$ is ample.
\end{theorem}

More interestingly, Theorem \ref{nef} can be used to study the Kodaira dimension and the numerical dimension of a smooth toroidal compactification without the need of passing to a finite \'etale cover. Recall that the numerical dimension $\nu(D)$ of a nef divisor $D$ is the largest integer $k$ such that $[D^{k}]\in H^{2k}(X,\qq)$ is not zero. For varieties with nef canonical divisor, we denote by $\nu(X):=\nu(K_{X})$. We can then prove the following.

\begin{proposition}\label{numerical}
Let $(X,D)$ be a smooth toroidal compactification of a complex hyperbolic manifold of dimension $n\geq3$. Then $\nu(X)\geq n-1$.
\end{proposition}

In general, the Kodaira and numerical dimensions of a nef divisor do not agree. Nevertheless, it is conjectured that they are always the same for canonical divisors. We refer the interested reader to the book \cite{Kol92} for the history and known partial results  concerning the following fundamental conjecture in algebraic geometry.   

\begin{conjecture}[Abundance] \label{abundance}
Let $X$ be a smooth variety with nef canonical divisor. Then $\kappa(X)=\nu(X)$.  
\end{conjecture}

If Conjecture \ref{abundance} holds true, Proposition \ref{numerical} implies the following result.

\begin{remark}\label{remark}
Assume Conjecture \ref{abundance}. Let $(X,D)$ be a smooth toroidal compactification of a complex hyperbolic manifold of dimension $n\geq3$. Then the Kodaira dimension of $X$ is bigger than or equal to $n-1$.
\end{remark}

Since the abundance conjecture is known to be true for $\dim(X)\leq 3$, see \cite{Kol92}, we can collect a proposition regarding threefolds which are smooth compactifications of ball quotients. Surprisingly, it seems we are still missing an explicit example of a smooth three dimensional compactification of a ball quotient. It is our hope 
that the next proposition will help in the construction of such examples. 

\begin{proposition}\label{three}
Let $(X,D)$ be a smooth toroidal compactification of a complex hyperbolic manifold of dimension $n=3$. Then $\kappa(X)\geq 2$.
\end{proposition}

Because of Remark \ref{remark} and the lack of counterexamples,  it seems interesting to ask if all smooth compactifications of higher dimensional ball quotients have ample canonical line bundle. 

\begin{question}\label{question}
Can we find a smooth toroidal compactification of a ball quotient with $dim(X)\geq 3$ and $K_{X}$ not ample?
\end{question}
\begin{remark}
In a recent preprint \cite{BT}, Bakker and Tsimerman made considerable progress towards a final answer of Question \ref{question}.
\end{remark}

The importance of Theorem \ref{nef} is not limited to the study of the Kodaira dimension of a smooth compactification. Recall in fact that in \cite{DiCerbo2}, we have extensively shown that a weaker form of Theorem \ref{nef} (see Theorem 1.1 in \cite{DiCerbo2}) can be successfully applied to study the geometry of finite volume complex hyperbolic manifolds. More precisely, it is shown how such theorem can be used to derive effective versions of classical results such as the Baily-Borel embedding theorem, Wang's finiteness theorem for complex hyperbolic manifolds, bounds on the numbers of cuspidal ends, effective very ampleness results for smooth toroidal compactifications and bounds on the Picard numbers of such compactifications.

Theorem \ref{nef} presented in this paper can be used to strengthen only one of the results previously derived in \cite{DiCerbo2}.
More precisely, we can prove a better upper bound on the number of cuspidal ends of a complex hyperbolic manifold in terms of its normalized Riemannian volume. For the precise statement of this result we refer to Theorem \ref{cusps} in Section \ref{Applications}. This bound is currently the best bound available in the literature in dimensions $3\leq n\leq 23$.  For $n\geq 24$ the bound derived by Hwang in \cite{Hwa1} is better, while for $n=2$ we have previously found a sharp result, see \cite{DiCerbo1}.\\

The paper is organized as follows. In Section \ref{hyperbolic}, we recall the theory of finite volume complex hyperbolic manifolds and their compactifications. Moreover, we add a new result showing that smooth toroidal compactifications of ball quotients are canonical, being unique up to biholomorphism. It seems that this result has been previously observed in the literature in dimension two only. Thus, we have decided to explicitly give a complete proof here. For the details see Proposition \ref{unico}. In Section \ref{bend}, we recall some foundational concepts and theorems from the theory of the minimal model, such as the bend and break theorem. Moreover, we recall some basic results regarding the numerical dimension of a nef divisor which are needed in the proof of Proposition \ref{numerical}. In Section \ref{nefness}, we provide the details of the proof of Theorem \ref{nef}. In Section \ref{Applications}, we collect quite different applications of Theorem \ref{nef}. First, we prove an upper bound on the number of cusps of a finite volume complex hyperbolic manifolds which improves in all dimension the one given in \cite{DiCerbo2}.  Second,  we give a proof of Proposition \ref{numerical}. More precisely, we use Theorem \ref{nef} to estimate the numerical dimension of the canonical divisor of a smooth toroidal compactification. Finally, we give a proof of Theorem \ref{Mum} which ensures that, up to a finite \'etale cover, all such compactifications have ample canonical class.

\vspace{0.5cm}

\noindent\textbf{Acknowledgements}. The first named author would
like to thank Professor J\'anos Koll\'ar and Roberto Svaldi for many constructive comments. The second author would like to thank Mike Roth for an useful discussion. He also thanks the Max Planck Institute for Mathematics for the hospitality and support during the final stages of this project.

\section{Preliminaries, Notations and a Uniqueness Result}\label{preliminaries}

\subsection{Hyperbolic manifolds, their compactifications and their uniqueness}\label{hyperbolic}

Let $\mathcal{H}^{n}$ be the complex $n$-dimensional hyperbolic
space of dimension $n\geq2$. Let $X^{o}$ be a metrically complete
non-compact complex hyperbolic manifold of finite volume. It is well known that $X^{o}:=\mathcal{H}^{n}/\Gamma$ where
$\Gamma$ is torsion-free lattice in $\textrm{PU}(n,1)$. Since $X^{o}$ is non-compact then $\Gamma$ must contain
parabolic elements. Then $X^{o}$ has finitely
many disjoint unbounded ends of finite volume called the cusps of
$X^{o}$.  It is known that when the parabolic elements in $\Gamma$ have no rotational part, then $X^{o}$ admits a compactification $(X, D)$ consisting of a smooth projective variety and an exceptional divisor $D$. Recall that each maximal parabolic subgroup can be thought as a lattice in $\textrm{H}\rtimes \textrm{U}(n-1)$ where $\textrm{H}$ is a Heisenberg type Lie group of real dimension $2n-1$, and that a parabolic isometry is said to have no rotational part if it has no $\textrm{U}(n-1)$ component. The pair $(X, D)$ is known as the toroidal compactification of $X^{o}$. For the detailed construction of the compactifications $(X, D)$, we refer to the book \cite{Ash} and to the paper \cite{Mok}. For a more detailed introduction, the interested reader may also refer to Section 1.1 in \cite{DiCerbo2}. 

Let us recall the geometric features of $(X, D)$ which will be needed in the remaining of this work. First, the pair $(X, D)$ is by construction a resolution of the Baily-Borel \cite{Borel1}  compactification $X^{*}$ of $X^{o}$. Recall that $X^{*}$ is a normal projective variety
such that the complement of $X^{o}$ in $X^{*}$ consists of only finitely many (singular) points, called cusp points. When $\Gamma$ in non-arithmetic, the compactification $X^{*}$ has been constructed in \cite{Siu}, see also \cite{Mok}. Moreover, the exceptional divisor $D$ consists of disjoint smooth Abelian varieties with negative normal bundle in $X$. Thus, the irreducible components of $D$ are in one-to-one correspondence with the cusps of $X^{o}$ or equivalently with the cusp points in $X^{*}$. Finally, we have nice positivity properties for the log-canonical divisor of the pair $(X, D)$. 

\begin{proposition}\label{positivity}
 Let $(X,D)$ be a smooth toroidal compactification of a complex hyperbolic manifold of dimension $n\geq2$.  Then $K_{X}+D$ is big, nef and strictly nef outside $D$.
\end{proposition} 

\begin{proof}
It can be shown that the standard locally symmetric K\"ahler-Einstein metric on $X^{o}$, when regarded as a current on $X$, is a strictly positive K\"ahler current with singular support exactly $D$. The proof then follows easily, for more details see \cite{DiCerbo2}. Alternatively, one can show that $K_{X}+D$ is the pull back of an ample line bundle on the Baily-Borel compactification $X^{*}$ via the map $\pi: X\rightarrow X^{*}$, see Proposition 3.4  in \cite{Mumford}. Finally, we refer to \cite{Yau} for a detailed study of the positivity properties of the log-canonical divisor of general smooth toroidal compactifications. Among many other things, in \cite{Yau} it is shown that $K_{X}+D$ can never be ample for smooth toroidal compactifications of Shimura varieties.
\end{proof}

From now on, we always assume the lattice $\Gamma\in \textrm{PU}(n,1)$ to be non-uniform, torsion free and with rotation free parabolic elements.

For most of the arguments presented in this work this is all we need to know on $(X, D)$. Nevertheless, for the proof of Theorem \ref{Mum} we need to discuss the cusps of $X^{o}$ a bit more and their ``filling'' in $(X, D)$. Thus, given $X^{o}$ as above let us denote by $ A_{1}$, ..., $A_{m}$ its cusps. Recall that the cusps are in one-to-one correspondence with maximal parabolic subgroups of $\Gamma$ say $\Gamma_{1}$, ..., $\Gamma_{m}$. Given any $A_{i}$, the horobal fixed by the corresponding $\Gamma_{i}$ can be identified with a Heisenberg type group $\textrm{H}_{i}$ so that $A_{i}$ is isomorphic to 
$\textrm{H}_{i}/\Gamma_{i}\times [0, \infty)$, where now $\textrm{H}_{i}/\Gamma_{i}$ is a nilmanifold since $\Gamma_{i}$ is a lattice of left translations in $\textrm{H}_{i}$. Concretely, $\textrm{H}_{i}/\Gamma_{i}$ is a non-trivial $S^{1}$-bundle over a complex $(n-1)$-dimensional torus say $D_{i}$. The union of those $D_{i}$'s is the divisor $D$ in the compactification $(X, D)$. For any $i$, let us observe that $\textrm{C}_{i}=[\textrm{H}_{i}, \textrm{H}_{i}]$
is the center of $\textrm{H}_{i}$ with $\textrm{C}_{i}$  isomorphic to $\rr$, so that the centers of the maximal parabolic subgroups $\Gamma_{i}$'s are lattices in $\rr$ generated by a single element say $\alpha_{i}\in[\Gamma_{i}, \Gamma_{i}]$ of minimal length. For much more on this construction we refer to Section 3 in \cite{Hummel} and to the bibliography therein. Concluding, we would like to point out that the theory developed \cite{Hummel} is independent of the arithmeticity of 
the lattices in $\textrm{PU}(n,1)$. This is of some interest since the arithmeticity assumption is crucially used in the constructions presented in \cite{Ash}. This technical point is also lucidly discussed in \cite{Mok}.\\

We conclude this section by addressing the uniqueness of smooth toroidal compactifications. Thus, let $X^{o}$ let be a complex hyperbolic manifold associated with a non-uniform rotation free lattice in $\textrm{PU}(n,1)$. The cusp closing construction of Hummel-Schroeder \cite{Hummel} and Mok \cite{Mok} then provides a compactification $X$ with compactifying divisor $D$ given by a disjoint union of abelian varieties. From this differential geometric point of view, it is not immediately clear that the pair $(X,D)$ is uniquely determined by $X^{o}$. Let us first observe that by construction the pair $(X, D)$ is a resolution of the singularities of $X^{*}$. Thus, if we assume $dim(X)=2$ it then follows that such a compactification must be unique. In fact, the exceptional divisor $D$ does not contain any rational curve and then $X$ is a minimal resolution of $X^{*}$. Recall that in complex dimension two any normal surface admits a unique minimal resolution, see for example Theorem 6.2 in \cite{Bar}. In higher dimensions the proof is somewhat different due to the lack of minimal resolution. The key result is the following lemma on rational maps and their locus of indeterminacy. 

\begin{lemma}\label{defined}
Let $Y$ be a smooth variety and let $g: Y \dashrightarrow X$ be a rational map. Let $Z\subseteq Y\times X$ be the closure of the graph of $g$ and let $p$ and $q$ be the first and second projections. Let $S\subset Y$ be the set where $g$ is not defined. Then for any point $z\in q(p^{-1}(S))$ there exists a rational curve $z\in C_{z}\subseteq q(p^{-1}(S))$.
\end{lemma}

\begin{proof}
We resolve the indeterminacy of $g$ by blowing up $Y$ along smooth centers. In particular, we are blowing up smooth subvarieties contained in $S$.  Each point $z\in q(p^{-1}(S))$ is dominated by one the exceptional divisor of the blowup. Since the map is not defined on $S$, there exists a curve $C$ in the exceptional divisor which is not contracted in $X$. Since rational curves map only to rational curves, we deduce the result. 
\end{proof}

For more on the standard theory of rational maps we refer to \cite{KM}. Next, we can use the lemma to deduce the uniqueness of such toroidal compactifications.
\begin{proposition}\label{unico}
Let $X^{o}$ be a hyperbolic manifold of $\dim(X)\geq 2$ which admits a smooth toroidal compactification $(X, D)$. The pair $(X, D)$ is uniquely determined by $X^{o}$.
\end{proposition}

\begin{proof}
Suppose $X^{o}$ admits two toroidal compactifications $(X_{1},D_{1})$ and $(X_{2},D_{2})$. Since $X^{o}$ is by construction biholomorphic to 
$X_{1}\backslash D_{1}$ and $X_{2}\backslash D_{2}$, there exists a birational map $g: X_{1} \dashrightarrow X_{2}$. Moreover, if $S$ is the locus 
where $g$ is not defined, then $S\subseteq D_{1}$ and $q(p^{-1}(S))\subseteq D_{2}$, where $p$ and $q$ are defined in Lemma \ref{defined}. 
In particular, $D_{2}$ must contain a rational curve which is impossible. This shows that $g$ is everywhere defined and the same argument applied to $g^{-1}$ shows that $g$ must be an isomorphism. 
\end{proof}

In conclusion, the cusp closing construction of Hummel-Schroeder and Mok produces a canonical compactification. Of course, it is still possible that two nonisomorphic finite volume complex hyperbolic manifolds when compactified produce the same smooth projective variety $X$. The difference is then in the compactifying divisors say $D_{1}$ and $D_{2}$. For very explicit examples we refer to \cite{LucaS}.

\subsection{Bend and break and numerical dimension}\label{bend}

In this section, we recall some basic results from the theory of the minimal model. For completeness we recall here the precise statements 
of the results we are going to use in Section \ref{nefness}. For more details and the proofs of such results, we refer to \cite{KM}.\\

The first and probably most important result for us is bend and break, see Lemma 1.9 in \cite{KM}.

\begin{lemma}[Bend and break] \label{bend}
Let $X$ be a normal projective variety and $g_{o}:\pp^{1}\rightarrow X$ a non-constant morphism. Assume that there is a smooth 
connected (possibly non-proper) pointed curve  $0_{C}\in C$ and a morphism $G:\pp^{1}\times C \rightarrow X$ such that 
\begin{enumerate}
\item $G|_{\pp^{1}\times \left\{0_{C} \right\}}=g_{o}$,
\item $G(\left\{0\right\}\times C)=g_{0}(0)$, $G(\left\{\infty\right\}\times C)=g_{0}(\infty)$ and 
\item $G(\pp^{1}\times C)$ is a surface. 
\end{enumerate}

Then $(g_{o})_{*}(\pp^{1})$ is algebraically equivalent to either a reducible curve or a multiple curve. 
\end{lemma}
%The bend and break theorem is crucially used in the proof Theorem \ref{nef}. To the best of our knowledge, this is the first instance were 
%bend and break techniques are applied to hyperbolic geometry. 

Next, we need a classification result for certain extremal contractions. Let us fix notation and be more precise. We denote by $\NE (X)$ the closure of the cone of effective 1-cycles in $X$ modulo numerical equivalence. Let $R$ be a $K_{X}$-negative extremal extremal ray in $\NE (X)$. In other words, given a curve $C$ in $X$ whose numerical class is such that $[C]\in R$,  then $K_{X}\cdot C <0$. We define the length of $R$ to be 
\bdism
 l(R):=\min\left\{-K_{X}\cdot C \:|\: \text{$C$ is a rational curve with numerical class in $R$}\right\}.
\edism
It follows from the cone theorem, see for example Theorem 1.24 in \cite{KM}, that $l(R)\leq n+1$, where $n$ is the dimension of $X$. Moreover, the same theorem implies that $K_{X}$-negative extremal rays 
can be contracted. More precisely, it ensures the existence of a projective variety $Y$ and morphism with connected fibers $\phi_{R}:X \rightarrow Y$, such that $\phi_{R}(C)$ 
is a point if and only if $[C]\in R$. Extremal contractions associated to extremal rays of low length have been successfully classified by Wi\'sniewski 
\cite{Wis2}. Before stating this result we need to recall that the dimension of the fibers of $\phi_{R}$ provides an upper bound 
on the length of the associated extremal ray. More precisely, in \cite{Wis2} it is proven the following:

\begin{theorem}[Wi\'sniewski]\label{length}
If $F$ is a nontrivial fiber of a contraction of $R$ then \bdism
\dim(F)\geq l(R)-1. \edism
\end{theorem}

We can now state the classification result for contractions of extremal rays with length at most two. 

\begin{theorem}[Wi\'sniewski]\label{conic}
Let $X$ be a smooth variety. Let $\phi_{R}: X\rightarrow Y$ be the
contraction of a $K_{X}$-negative extremal ray $R$ of $X$ such that $\dim(F)\leq 1$.
Then $Y$ is smooth and either
\begin{enumerate}
\item $\phi_{R}: X\rightarrow Y$ is a conic bundle, or
\item $\phi_{R}: X\rightarrow Y$ is the blow-up of the variety $Y$ along a smooth subvariety $Z$ of codimension $2$.
\end{enumerate}
\end{theorem}

Theorem \ref{conic} and bend and break are the main technical tools used in the proof of Theorem \ref{nef}. \\

We conclude this section by recalling the definition of the numerical dimension of a nef divisor.

\begin{definition}\label{ndim}
Let $D$ be a nef divisor. Its numerical dimension $\nu(D)$ is defined as 
\bdism
\nu(D):=\max_{k\in\nn}\left\{D^{k}\cdot A^{n-k}>0 \right\}.
\edism
where $A$ is any ample divisor. 
\end{definition}

It is easy to see that this definition does not depend on the choice of the ample line bundle $A$. Moreover, we recall that 
$\nu(D)$ is the greatest integer $k$ such that $[D^{k}]$ is not trivial in $H^{2}(X,\qq)$. 

The numerical dimension of a nef divisor is closely connected to its Kodaira dimension. It is possible to show that in general the following inequality holds $\kappa(D)\leq \nu(D)$. On the other hand, the numerical dimension is a better suited invariant, for example it is an invariant of the numerical class of $D$. For more information on these numerical invariants, we refer the reader to \cite{Lehmann} and the bibliography therein.

Finally, we recall that if $K_{X}$ is a nef divisor then it is expected that numerical dimension and Kodaira dimension agree. This statement is equivalent to the abundance conjecture and, in particular, to the existence of good minimal models, see \cite{GL} for more details.   

\section{Nefness of the canonical divisor}\label{nefness}

In this section, we prove Theorem \ref{nef}. Regarding the organization of the proof, we first address the nefness of the canonical class $K_{X}$ and then prove that $K_{X}+\alpha D$ is an ample $\rr$-divisor for all $\alpha\in (0, 1)$.

\begin{proof}[Proof of Theorem \ref{nef}] 

We first show that $K_{X}$ must be a nef divisor. By contradiction, let us assume this is not the case. 
Thus, there exists at least one $K_{X}$-negative extremal ray in $\NE (X)$. Let $R$ be such extremal ray in $\NE (X)$ and let $\phi_{R}: X\rightarrow Y$ be the associated contraction. Let us denote by $F$ an irreducible 
component of a non-trivial fiber of $\phi_{R}$. Let $C$ be an irreducible curve in $F$. 
Since $C$ is contracted by the contraction of $R$, the contraction theorem \cite{KM}, gives that $[C]\in R$. By Proposition \ref{positivity}, we know $K_{X}+D$ is nef. Since $R$ is $K_{X}$-negative, we have that $D\cdot C>0$. Let us denote by $\{D_{i}\}$ the smooth irreducible components of $D$. If $C\subseteq D_{i}$ for some $i$,
we conclude that $D\cdot C<0$ as the normal bundle of each component $D_{i}$ is
negative. In particular, by dimension counting we must have $\dim(D\cap F)=0$ which then implies
$\dim(F)\leq 1$. Finally, Theorem \ref{length} implies that  $l(R)\leq 2$. In conclusion, it is proved that any $K_{X}$- negative 
extremal ray in a smooth toroidal compactification of a ball quotient has length at most two. 

By Theorem \ref{conic}, we have that the extremal contraction produces a smooth variety $Y$
and we have the following possibilities for $\phi_{R}$:
\begin{enumerate}
\item $\phi_{R}: X\rightarrow Y$ is a conic bundle, or
\item $\phi_{R}: X\rightarrow Y$ is the blow-up of the variety $Y$ along a smooth subvariety $Z$ of codimension $2$.
\end{enumerate}

Instead of working on $X$, we will pass to the Baily-Borel compactification $X^{*}$. Recall that $X^{*}$ is the normal variety
obtained from $X$ contracting the components of the divisor $D$. We denote by $\pi:X\rightarrow X^{*}$ the 
contraction map. 

First assume that $\phi_{R}$ is a conic bundle. Let $C$ be a smooth curve in $Y$ and let $F\cong \pp^{1}$ be a smooth fiber of $\phi_{R}$ over a point $0_{C}$ contained in $C\subseteq Y$. Let $S$ be the ruled surface in $X$ over $C$. Replacing $C$ by an open set we can assume $S=\pp^{1}\times C$. In particular, we can define $G:\pp^{1}\times C \rightarrow X^{*}$ restricting $\pi$ to $S$. Recall that $F$ cannot be contained in $D$ as its smooth irreducible components are Abelian varieties. Since by construction the complement of $D$ in $X$ is hyperbolic, we must have that $F$ intersects $D$ in at least three distinct points. Thus, the points where $F$ meets $D$ determine different fixed  points for the family of rational curves defined by $G$. By Theorem \ref{bend} there must be a reducible fiber over $C$. Note that this implies that $\phi_{R}$ must have singular fibers. Since the discriminant locus of a conic bundle is a divisor on the base, see for example \cite{Sar82}, it has dimension at least one by our assumption. We are free to choose $C$ entirely contained in the discriminant locus. In particular, the irreducible components of the fibers over $C$ define a new family of rational curves on $X^{*}$, as we did before, with at least two fixed points. Applying again Theorem \ref{bend} we obtain that $\phi_{R}$ has a fiber with three irreducible components. This is a contradiction because every fiber of $\phi_{R}$ is isomorphic to a conic in $\pp^{2}$. 

It remains to show that $\phi_{R}$ cannot be the blow up along a smooth subvariety $Z\subseteq Y$ of codimension $2$. Let $E$ be the exceptional divisor of $\phi_{R}$ and recall that $E$ is a $\pp^{1}$-bundle over $Z$. Since we are assuming $dim(X)\geq 3$, we can always find a smooth curve $C\subseteq Z$. Let $F\cong \pp^{1}$ be a smooth fiber of $\phi_{R}$. By  eventually replacing $C$ with an open set, we can assume $S=\pp^{1}\times C$ and define a family of rational curves $G:\pp^{1}\times C\rightarrow X^{*}$ via the map $\pi$. Since by construction the complement of $D$ in $X$ is hyperbolic, we must have that $F$ intersects $D$ in at least three distinct points. Then by bend and break, $\phi_{R}$ must have a singular fiber. This is a contradiction because any non-trivial fiber of $\phi_{R}$ is a smooth curve. We showed that there are no negative extremal rays and thus $K_{X}$ is a nef divisor. 

To conclude the proof, we need to show that $K_{X}+\alpha D$ is an ample $\rr$-divisor for any $\alpha\in (0,1)$. We claim that it is enough show 
that $K_{X}+\alpha D$ is ample for all values of $\alpha$ close to one. In fact as shown in the first part of this proof, $K_{X}$ is always inside the closure of the ample cone. Now, the ample cone is convex so that if $K_{X}+\alpha D$ is ample for all $\alpha\in(1-\epsilon, 1)$ for some $\epsilon>0$, then it is necessarily ample for all $\alpha\in (0, 1)$. 

Let us show that $K_{X}+\alpha D$ is ample for all $\alpha\in(1-\epsilon, 1)$ for some $\epsilon>0$. Recall that because of  Proposition \ref{positivity} $K_{X}+D$ is big and nef. By Theorem 4.15 in \cite{DiCerbo}, we need to show that there are no curves $C$ in $X$ such that $(K_{X}+D)\cdot C=0$ and $K_{X}\cdot C\leq0$. Suppose $C$ is a curve such that $(K_{X}+D)\cdot C=0$. By Proposition \ref{positivity}, we know 
that $K_{X}+D$ is striclty nef outside $D$ which implies that $C\subseteq D$. On the other hand, the normal bundle of $D$ in $X$ is negative  which forces $K_{X}\cdot C=-D\cdot C>0$. The argument is complete.
\end{proof}

\begin{remark}
The proof of Theorem \ref{nef} fails in dimension two. The key point is that the blow-up operation for surfaces does not define a family of 
rational curves but rather a unique rigid curve. 
\end{remark}

\section{Applications}\label{Applications}

In this final section we collect some quite different applications of Theorem \ref{nef}. Let us start by showing that ampleness range given in Theorem \ref{nef} can be used to give an upper bound on the number of cuspidal ends of a complex hyperbolic manifold in terms of its Riemannian volume. This bound improves in all dimensions the one given in \cite{DiCerbo2}, see Theorem 1.5 therein. Recall that the Riemannian volume of the hyperbolic manifold $X^{o}$ can be computed in terms of the top self-intersection of the log-canonical of its smooth compactification $(X, D)$. More precisely, if we normalize the holomorphic sectional curvature to be $-1$, one has
\begin{align}\notag
\Vol(X^{o})= \frac{(4\pi)^{n}}{n!(n+1)^{n}}(K_{X}+D)^{n},
\end{align}
where $n$ is the dimension of $X$. Since the number of cusps are in one-to-one correspondence with the irreducible components of $D$, we can then state the theorem in the following form which is consistent with the results contained in \cite{DiCerbo2}.

\begin{theorem}\label{cusps}
Let $(X,D)$ be a toroidal compactification with $\dim(X)\geq 3$. Let $q$ be the number of irreducible components of $D$. 
Then
\bdism
q\leq (K_{X}+D)^{n}.
\edism
\end{theorem}

\begin{proof}
Let $L=K_{X}+D$. By adjunction $L_{|_{D}}=\mathcal{O}_{D}$, in other words $L$ restricted to $D$ is the trivial line bundle. Thus, for any $0<i<n$ we have the vanishing of the intersection numbers $L^{i}\cdot D^{n-i}$. Moreover, by Theorem \ref{nef}, we know 
that $2L-D$ is an ample divisor. In particular, 
\begin{align}\label{count}
q\leq D\cdot (2L-D)^{n-1}=(-1)^{n-1}D^{n}.
\end{align}
By Theorem \ref{nef}, we know that $K_{X}$ is nef  which then implies $K^{n}_{X}\geq 0$. Thus, we have
\begin{align}\notag
0\leq K^{n}_{X}=(L-D)^{n}=L^{n}+(-1)^{n}D^{n}
\end{align} 
which combined with (\ref{count}) gives
\begin{align}\notag
q\leq (-1)^{n-1}D^{n}\leq (K_{X}+D)^{n}.
\end{align} 
The proof is then complete.\\
\end{proof}

\begin{remark}
It is interesting to observe that the statement of Theorem \ref{cusps} is false for surfaces. In fact, Hirzebruch constructed a smooth compactification of a complex hyperbolic surface with four cusps such that $(K_{X}+D)^{2}=3$, see again \cite{Hirzebruch}.
\end{remark}

Remarkably, the simple computations given in the proof of Theorem \ref{cusps} can be also used to obtain a lower bound on the numerical dimension of $X$.

\begin{proof}[Proof of Proposition \ref{numerical}] 

We need to compute intersection products of $K_{X}$ with any fixed ample line bundle $A$. 
As shown in Theorem \ref{nef}, for any smooth compactification $(X, D)$ the divisor $2K_{X}+D$ is ample. Let us then choose $A=2K_{X}+D$, and in order to simplify the notation let us define $L:=K_{X}+D$. Thus, for any $0 \leq k\leq n$ we compute 
\bdism
K_{X}^{k}\cdot (2K_{X}+D)^{n-k}=(L-D)^{k}\cdot (2L-D)^{n-k}=2^{n-k}L^{n}+(-1)^{n}D^{n}.
\edism
Again by Theorem \ref{nef}, $K_{X}$ is always nef so that
\begin{align}\notag
0\leq K^{n}_{X}=L^{n}+(-1)^{n}D^{n}.
\end{align}
Combining these two inequalities we get 
\bdism
K_{X}^{k}\cdot (2K_{X}+D)^{n-k}\geq (2^{n-k}-1)L^{n},
\edism 
where the right hand side is strictly positive as long as $k\neq n$.  By Definition \ref{ndim}, we obtain that $\nu(X)\geq n-1$.\\
\end{proof}

\begin{remark}
Proposition \ref{numerical} is false in dimension $2$. In fact, some of the two dimensional smooth compactifications constructed in \cite{Hirzebruch} have zero numerical dimension.\\
\end{remark}

We conclude this section by showing that $K_{X}$ is ample up to a finite \'etale cover, see the statement of Theorem \ref{Mum}. The ampleness of $K_{X}$ is achieved by bootstrapping the positivity of the canonical divisor along a tower of coverings. In dimensions $n\geq 3$, the key for this argument is again Theorem \ref{nef}. In other words, the input for the bootstrap process is the nefness the canonical divisor at the base of the tower of coverings. For $n=2$, the argument cannot be performed because we cannot in general assume the base of the tower to have nef canonical divisor.  Nevertheless, the result still holds true as originally shown in \cite{DiC}.

\begin{proof}[Proof of Theorem \ref{Mum}] 

If the dimension of $X$ is $n=2$, we refer to Theorem A in \cite{DiC}. We also refer to the bibliography in \cite{DiC} for earlier partial results in this direction. Thus, from now on we assume $n\geq 3$.
Given a torsion free lattice $\Gamma\subseteq \textrm{PU}(n,1)$, let us consider a finite index subgroup $\Gamma'\leq \Gamma$ whose parabolic isometries have no rotational part. This is always possible as shown for example in \cite{Hummel2}. If $\Gamma$ is assumed to be arithmetic, then $\Gamma'$ can be any neat finite index subgroup, see for example \cite{Ash}. Thus $\mathcal{H}^{n}/\Gamma'$ admits a smooth toroidal compactification say $(X', D')$. Let us denote by $\Gamma_{i}$, say for $i=1, ..., m$, the maximal parabolic subgroups of $\Gamma'$. Recall that each $\Gamma_{i}$ is a co-compact torsion free lattice of left translations on the corresponding  horosphere  $\textrm{H}_{i}$. For any $i$, let us denote by $\alpha_{i}$ the generator of $\Gamma_{i}\cap C_{i}$, where $C_{i}$ is the center of $\textrm{H}_{i}$. Let us then consider the set of parabolic isometries $\mathcal{P}=\{\alpha_{1}, ..., \alpha^{l}_{1}, ..., \alpha_{m}, ..., \alpha^{l}_{m}\}$, where $l$ is a fixed integer $\geq 1$. Since any lattice in $\textrm{PU}(n,1)$ is residually finite, let us consider a finite index subgroup $\Gamma_{0}\leq \Gamma'$ such that $\Gamma_{0}\cap\mathcal{P}=\{0\}$. The finite \'etale map $f: \mathcal{H}^{n}/\Gamma_{0}\rightarrow \mathcal{H}^{n}/\Gamma'$ now extends to a branched covering map $p: (X_{0}, D_{0})\rightarrow (X', D')$, where $(X_{0}, D_{0})$ is a toroidal compactification of $\mathcal{H}^{n}/\Gamma_{0}$. The claim is now that $p$ branches along the whole divisor $D'$.  This follows from the fact that by construction the generators of the centers of the maximal parabolic subgroups in $\Gamma_{0}$ cover  at least of degree two the generators of the centers in $\Gamma'$. 

By the Hurwitz formula we have that $K_{X_{0}}=p^{*}K_{X'}+R_{0}$, where $R_{0}$ is possibly a non-reduced divisor whose support coincides with $D_{0}$. We claim that $K_{X_{0}}$ is necessarily strictly nef and big. Let us denote by $\{D_{i}\}$ the components of $D_{0}$. We then know that $R=\sum_{i}(r_{i}-1)D_{i}$ where by construction we have $r_{i}\geq 2$ for every $i$. Next, let us observe that because of Theorem \ref{nef}
\begin{align}\notag
K_{X_{0}}-\sum_{i}(r_{i}-1)D_{i}=p^{*}K_{X'}
\end{align}
is nef. Therefore, we have that
\begin{align}\notag
(p^{*}K_{X'})^{n}=(L-\sum_{i}r_{i}D_{i})^{n}=L^{n}+\sum_{i}r^{n}_{i}(-1)^{n}D^{n}_{i}\geq 0
\end{align}
where $L=K_{X_{0}}+\sum_{i}D_{i}$. Since $r_{i}\geq 2$ for any $i$, we then conclude that:
\begin{align}\notag
L^{n}>\sum_{i}(-1)^{n-1}D^{n}_{i}
\end{align}
which then implies that 
\begin{align}\notag
(K_{X_{0}})^{n}=(L-\sum_{i}D_{i})^{n}=L^{n}+\sum_{i}(-1)^{n}D^{n}_{i}>0.
\end{align}
Next, let us show that $K_{X_{0}}$ is strictly nef. By Theorem \ref{nef} the $\rr$-divisor $K_{X_{0}}+\alpha D_{0}$ is ample for any $\alpha\in (0, 1)$, thus for any curve $C\subset X_{0}$ which is entirely contained in $X_{0}\backslash D_{0}$ we have 
\begin{align}\notag
(K_{X_{0}}+\alpha D_{0})\cdot C=K_{X_{0}}\cdot C>0.
\end{align}
If the curve $C$ is now contained in $D_{0}$, we have
\begin{align}\notag
K_{X_{0}}\cdot C=-D_{0}\cdot C>0.
\end{align}
Finally, if $C$ is not contained in $D_{0}$ but it does intersect at least one of it s irreducible components we have
\begin{align}\notag
K_{X_{0}}\cdot C=(p^{*}K_{X'}+\sum_{i}(r_{1}-1)D_{i})\cdot C\geq \sum_{i}(r_{1}-1)D_{i}\cdot C>0.
\end{align}
Since $K_{X_{0}}$ is proven to be strictly nef and big, the base point free theorem, see Theorem 3.3 in \cite{KM}, implies that $K_{X_{0}}$ is indeed ample. For more details see for example Corollary 3.8 in \cite{DiCerbo}.

Next, it remains to show that for any $\Gamma_{1}\leq \Gamma_{0}$ the associated compactification $(X_{1}, D_{1})$ has ample canonical class.
Following the previous argument we have a finite map $p_{1}: (X_{1}, D_{1})\rightarrow (X_{0}, D_{0})$ which in general may or may not branch along $D_{0}$.
By the Hurwitz formula, we know that $K_{X_{1}}=p^{*}_{1}K_{X_{0}}+R_{1}$ where $R_{1}$ is a possibly non-reduced divisor whose support is contained in $D_{1}$. Now, $K_{X_{0}}$ is ample so that $p^{*}_{1}K_{X_{0}}$ intersect positively with any curve not entirely contained in $D_{1}$. The previous argument then gives that $K_{X_{1}}$ is ample. 
\end{proof}

%\section*{Acknowledgements} 
%The first named author would like to thank Professor J\'anos Koll\'ar and Roberto Svaldi for many constructive %comments. The second author would like to thank Mike Roth for an useful discussion. He also thanks the Max Planck %Institute for Mathematics for the hospitality and support during the final stages of this project.

\end{document}